\documentclass[12pt]{amsart}
\usepackage[utf8]{inputenc}
\usepackage{amsfonts}
\usepackage{amsmath}
\usepackage{amsthm}
\usepackage{amssymb}
\usepackage{latexsym}
\usepackage{enumerate}
\usepackage{centernot}

\newtheorem{theorem}{Theorem}%[section]
\newtheorem{lemma}[theorem]{Lemma}
\newtheorem{proposition}[theorem]{Proposition}
\newtheorem{corollary}[theorem]{Corollary}

\newtheorem{question}[theorem]{Question}
\newtheorem{example}[theorem]{Example}
\newtheorem{definition}[theorem]{Definition}

\newcommand{\norm}[1]{\Vert#1\Vert}
\newcommand{\F}{\mathcal{F}}
\newcommand{\G}{\mathcal{G}}
\newcommand{\GM}{\mathcal{G}^{MAX}}
\newcommand{\Sh}{\mathcal{S}}
\newcommand{\FM}{\mathcal{F}^{MAX}}
\newcommand{\SM}{\mathcal{S}^{MAX}}
\newcommand{\N}{\mathbb{N}}
\newcommand{\R}{\mathbb{R}}
\newcommand{\K}{\mathbb{K}}
\newcommand{\C}{\mathbb{C}}

\title{Isometries of combinatorial Banach spaces}
\author{C. Brech}
\address{Departamento de Matem\'atica, Instituto de Matem\'atica e Estat\'\i stica, Universidade de S\~ao Paulo, Rua do Mat\~ao, 1010 - CEP 05508-090 - S\~ao Paulo - SP - Brazil}
\email{brech@ime.usp.br}
\author{V. Ferenczi}
\address{Departamento de Matem\'atica, Instituto de Matem\'atica e Estat\'\i stica, Universidade de S\~ao Paulo, Rua do Mat\~ao, 1010 - CEP 05508-090 - S\~ao Paulo - SP - Brazil}
\email{ferenczi@ime.usp.br}
\author{A. Tcaciuc}
\address{Department of Mathematics and Statistics, MacEwan University, 10700-104 Avenue Edmonton, Alberta, T5J 4S2, Canada  }
\email{tcaciuca@macewan.ca}
\thanks{The first author was supported by CNPq grant (308183/2018-5), and the second author by CNPq grant (303034/2015-7). The first and second authors were supported by FAPESP grant (2016/25574-8). The third author was supported by MacEwan Project grant (01891).}

\subjclass[2010]{46B04, 46B45, 03E05, 03E75}

\keywords{regular families, combinatorial spaces, isometry group, Schreier families}

\date{\today}

\begin{document}
\baselineskip 18pt
\maketitle

\begin{abstract}
We prove that every isometry between two combinatorial spaces is determined by a permutation of the canonical unit basis combined with a change of signs. As a consequence, we show that in the case of Schreier spaces, all the isometries are given by a change of signs of the elements of the basis. Our results hold for both the real and the complex cases.
%We also get that the isometry groups of combinatorial spaces are light.
\end{abstract}

\section{introduction}

Classical results guarantee that every isometry of the spaces $c_0$ or $\ell_p$, $1 \leq p < \infty$, $p\neq 2$, are determined by a permutation of the elements of the canonical unit basis and a change of signs of these vectors (see e.g. \cite[Theorem 9.8.3 and Theorem 9.8.5]{ro}). Recently, it has been shown by Antunes, Beanland and Viet Chu \cite{abh} that the real Schreier spaces of finite order have a more rigid structure: isometries of these spaces correspond to a change of signs of the elements of the canonical unit basis. In this paper we generalize these results to higher order Schreier spaces and more general combinatorial spaces, in both the real and the complex cases, answering a question posed by K. Beanland in a private conversation. We also characterize the isometries that may arise between different combinatorial spaces.

In what follows, we consider spaces with either real or complex scalars. In this context, we call a scalar of modulus $1$ a {\em sign} (so simply $\pm 1$ in the real case). %Let $\mathbb{K}$ be either $\mathbb{R}$ or $\mathbb{C}$ and 
Recall that for a given regular family $\F$ (i.e. hereditary, compact and spreading, see Definition \ref{regular}) of finite subsets of $\N$, the combinatorial Banach space $X_\F$ is the completion of $c_{00}$, the vector space of finitely supported scalar sequences, with respect to the norm:
$$
\norm{x}=\sup\left\{\sum_{i\in F}|x(i)|: F\in\F\right\}.
$$
The sequence of unit vectors $(e_n)_n$ forms an unconditional Schauder basis, and $X_\F$ is $c_0$-saturated, see \cite{fonf}, so in particular it contains no copies of $\ell_1$. Therefore the basis $(e_n)_n$ is shrinking (see Theorem 1.c.9 in \cite{lt}), hence $(e_n^*)_n$ is a Schauder basis of the dual space $X_\F^*$. 

The most simple examples of regular families are the families $[\N]^{\leq n}$ of all subsets of $\N$ of cardinality at most $n$ for some fixed $n \in \N$. More interesting examples are the Schreier family
$$\mathcal{S}:= \{F \in [\N]^{< \omega}: |F| \leq \min F\} \cup \{ \emptyset\}$$
and its versions of higher order, which will be considered in Section \ref{secSchreier}. 

Given two combinatorial spaces $X_\F$ and $X_\G$ and a surjective isometry $T: X_\F^* \rightarrow X_\G^*$ between the corresponding dual spaces, we can use the classical fact that extreme points of the unit balls are preserved by $T$ to analyze the expansion of each $Te_i^* = \sum_j \alpha^i_j e_j^*$. This is the analysis we make in Section \ref{secIsom} to prove our main result (Theorem \ref{perm}), which states that if $\F$ and $\G$ are regular families, then $T$ is induced by what we call a {\em signed permutation}, i.e. for every $i \in \N$, $Te_i^* = \theta_i e^*_{\pi(i)}$ for some permutation $\pi: \N \rightarrow \N$ and some sequence of signs $(\theta_i)_i$. Since the adjoint operator of an isometry is an isometry, it follows in particular that any isometry $T: X_\F \rightarrow X_\G$ is also induced by a signed permutation. 

Together with the fact that surjective isometries between Banach spaces preserve extreme points of the unit balls,  we are going to use extensively in our arguments the following description of the extreme points of the dual ball of a combinatorial space:
$$
Ext(X_{\F}^{*})=\left\{\sum_{i\in F}\theta_i e^*_i : F\in\FM \text{ and }(\theta_i)_{i \in F}\text{ is a sequence of signs} \right\},
$$
where $\FM$ denotes the family of maximal elements of $\F$ with respect to inclusion. Since this characterization of extreme points holds in the real case (see \cite{abh, gow}) but was not known in the complex case, we shall give a proof for complex combinatorial spaces, see Proposition \ref{extcharact}. Our proof includes the real case and seems to be simpler than the original proof. 

%Let us observe that the dual of any real combinatorial space has the CRSP (see \cite{abh}, \cite{lendro}), which means that any element of the unit ball is a sequentially convex combination of extreme points.
%So in particular the convex hull of extreme points is norm dense in the unit ball and this implies that any map between duals of combinatorial spaces sending extreme points to extreme points must be a contraction, a fact that we may use later. DO WE???? (Kika: I DON'T SEE WHERE.)
%Some of our results about operators between combinatorial spaces will be more general, relying %only on the hypothesis that extreme points are sent to extreme points.

Classical examples of combinatorial spaces are the spaces $X_{\Sh_\alpha}$ associated to the so called generalized Schreier families $S_\alpha$, for $\alpha < \omega_1$. As a consequence of our main result and specific properties of these families, we prove in Section \ref{secSchreier} that any isometry $T$ of $X_{\Sh_\alpha}^*$ acts on the canonical unit basis as a change of signs, that is, $Te_i^* = \theta_i e_i^*$ for some sequence of signs $(\theta_i)_i$.

%Among our final remarks in Section \ref{secFinal}, we show that the isometry groups of combinatorial Banach spaces and their duals are light (see Definition \ref{light}), a notion introduced by Megrelishvili in \cite{meg}. 

%\subsection*{Acknowledgement} We would like to thank J. Lopez-Abad for hints on the proof of Proposition \ref{extcharact}.

\subsection*{Acknowledgement} We would like to thank Jordi Lopez-Abad for helpful conversations and feedback, and in particular for the ideas leading to the proof of Proposition \ref{extcharact}.

\section{Preliminaries}

We start with the combinatorial background for our results.

\subsection{Regular families}

Let $[\mathbb{N}]^{< \omega}$ denote the family of all finite subsets of $\mathbb{N}$ and by a family we mean always a family of finite subsets of $\mathbb{N}$ which contains all singletons. We denote by $\FM$ the family of maximal elements of a family $\mathcal{F}$ with respect to inclusion. 

\begin{definition} \label{regular}
We say that a given family $\mathcal{F}$ is regular if it satisfies the following three conditions:
\begin{itemize}
     \item $\mathcal{F}$ is hereditary (closed under subsets);
    \item $\mathcal{F}$ is compact as a subset of $2^\mathbb{N}$, where each element of $\mathcal{F}$ is identified  with its characteristic function;
    \item $\mathcal{F}$ is spreading, that is, if $F \in \mathcal{F}$ and $\sigma:F \rightarrow \mathbb{N}$ is such that $\sigma(n) \geq n$ for every $n \in F$, then $\sigma(F) \in \mathcal{F}$.
\end{itemize}
\end{definition}

An easy property shared by regular families and which will be frequently used is the fact that any element can be extended ``to the right'' to a maximal one:

\begin{lemma}\label{propRegular}
If $\F$ is a regular family and $F \in \F$, then for every infinite set $N \subseteq \N$, there is $F \subseteq E \in \FM$ such that $E \setminus F \subseteq N$.
\end{lemma}
\begin{proof}
Given $F \in \F \setminus \FM$ and $N \subseteq \N$ infinite let $F \subsetneq E_1 \in \F$ and spread $E_1$ to some $F_1 $ in such that a way that $F \subsetneq F_1 \in \F$ and $F_1\setminus F \subseteq N$. If $F_1 \in \FM$ we are done. If not, let $F_1 \subsetneq E_2 \in \F$ and spread $E_2$ to some $F_2 $ in such that a way that $F_1
 \subsetneq F_2 \in \F$ and $F_2\setminus F_1 \subseteq N$, so that $F_2\setminus F \subseteq N$. Repeat this process until getting some $F \subseteq F_n \in \FM$ such that $F_n \setminus F \subseteq N$. This will necessarily happen, as if not, $(F_n)_n$ will be a strictly increasing chain of elements of $\F$ converging to the infinite set $Y = \bigcup_{n \in \N} F_n \notin \F$, contradicting the compactness of $\F$.
\end{proof}

\begin{lemma} \label{spreading}
Suppose $\F$ is a regular family, and let $n\in\N$ with $\{n\}\notin\FM$. Then we can find a sequence of finite sets $n<G_1<G_2<...$ such that for any $i\in\N$, $|G_i|\leq|G_{i+1}|$ and $G_i\cup \{n\}\in \FM$.
\end{lemma}

\begin{proof}
By Lemma \ref{propRegular}, we can find $F_1\in\FM$ such that $n=\min F_1$. Clearly $|F_1|\geq 2$, and let $G_1:=F_1\setminus\{n\}$. Using that $\F$ is spreading, we can find $F'_2\in\F$ such that $|F_1|=|F'_2|$,  $n=\min F'_2$, and  $G_1<F'_2\setminus\{n\}$. Next, from Lemma \ref{propRegular}, it follows that we can ``fill in" $F'_2$ to the right, if necessary, to obtain a set $F_2\in\FM$. Let $G_2:=F_2\setminus\{n\}$, and clearly  $|G_1|\leq|G_2|$. Continuing in this manner we obtain the conclusion of the lemma.
\end{proof}

\subsection{Extreme points in the dual space}

Denote by $\K$ the field of scalars $\R$ or $\C$. We say that a subset $N$ of a Banach space $X$ is {\em sign invariant} if for any sign $\theta\in\K$, we have $\theta N=N$.
We recall a very classical lemma in its real/complex version.

\begin{lemma} \label{norming}

Let $X$ be a Banach space over $\K$, and let $N\subseteq B_{X^*}$ be a sign invariant norming set for $X$. Then
% in the sense that  for any $x\in X$, $\norm{x}=\sup_{f\in N}|f(x)|$. 
$$
B_{X^*}=\overline{conv(N)}^{w*}.
$$

\end{lemma}

\begin{proof}
Denote $S:=\overline{conv(N)}^{w*}$ and note that $S$ is sign invariant. Assume by contradiction that the conclusion is false and pick $f\in B_{X^*}\setminus S$. $S$ is convex, $w^*$-compact and disjoint from $\{f\}$ (which is convex and $w^*$-closed). From Hahn-Banach separation theorem we have that there exists $x\in X$ 
%(here we use that the dual of $X^*$ with the $w^*$-topology is isomorphic to $X$), 
and a real number $t$ such that   
$$
  Re(g(x))<t<Re(f(x)), \mbox{ for all } g\in S.
$$
Multiplying a given $g \in S$ by the appropriate sign we may assume that 
%For a fixed $g\in X$, let $\delta:=\frac{\overline{g(x)}}{|g(x)|}$. Since $|\delta|=1$ we have  g\in 
%$ and it is easy to see that
 $Re( g(x))=|g(x)|$. Since $S$ is sign invariant, we have 
$$
 |g(x)|<t<Re(f(x)), \mbox{ for all } g\in S.
$$
Taking the supremum over all $g\in N$, we obtain the contradiction 
$$
\norm{x}<Re(f(x))\leq|f(x)|\leq \norm{f}\norm{x}\leq\norm{x},
$$
which finishes the proof. 
\end{proof}

%\begin{theorem}[Hahn-Banach separation theorem]
%Let $X$ be a locally convex  TVS and $A$ and $B$ bounded, convex, disjoint, subsets of $X$ %such that $A$ is closed and $B$ compact. Then there exists a bounded functional $f\in X^*$, %and real numbers $t<s$ such that 
%$$
 % Re(f(a))<t<s<Re(f(b)), \mbox{ for all } a\in A \mbox{ and } b\in B
%$$
%\end{theorem}

%We are going to use Milman's Theorem (see Rudin's Functional Analysis, Theorem 3.25, page %76).

%\begin{theorem}[Milman's Theorem]
%If $C$ is a compact set in a locally convex space $X$, and if $\overline{conv(C)}$ is also %compact, then every extreme point of $\overline{conv(C)}$ lies in $C$.
%\end{theorem}

\begin{proposition}\label{extcharact}
If $\F$ is a regular family, then 
$$
Ext(X_{\F}^{*})=\left\{\sum_{i\in F}\theta_i e^*_i : F\in\FM \text{ and }(\theta_i)_{i \in F}\text{ is a sequence of signs} \right\},
$$
\end{proposition}

\begin{proof}
Let $N:=\{\sum_{i\in F}\theta_i e^*_i : F\in \FM, |\theta_i|=1\}$ and $M:=\{\sum_{i\in F}\theta_i e^*_i : F\in \F, |\theta_i|=1\}$. Note that $M$ is norming for $X_{\F}$ and sign invariant.  From Lemma \ref{norming} it follows that  $B_{X_\F^*}=\overline{conv(M)}^{w^*}$.
 We claim that $M$ is $w^*$- closed. Then both $M$ and $B_{X_\F^*}=\overline{conv(M)}^{w*}$ are compact in the locally convex space $(X_\F^*,w^*)$, so by Milman's theorem (see \cite{rudin}, Theorem 3.25), every extreme point of $B_{X_\F^*}$ lies in $M$.  
%therefore $C:=\overline{N}^{w*}\subseteq M$. 
%$C$ as defined above is a $w^*$-closed subset of $B_{X^*}$, which is $w^*$-compact. %Therefore $C$ is $w^*$-compact. Also $\overline{conv(C)}^{w^*}$ is $w^*$-compact.   From %Milman's Theorem it follows now that $Ext(B_{X^*})\subseteq C$. Hence:
%$$
 % Ext(B_{X^*})\subseteq \overline{N}^{w*}\subseteq M
%$$

Since $N\subseteq Ext(B_{X_\F^*})$ and no $x\in M\setminus N$ is an extreme point (any such $x$ is easily written as the middle point of two different points of $N$), it follows that $Ext(B_{X_\F^*})=N$. 

To prove the claim we note that if a sequence of points of $M$ converges $w^*$ to some $y$, then the compactness of $\F$ imples that the support of $y$ belongs to $\F$. The $w^*$ convergence implies coordinatewise convergence, and so each non-zero coordinate of $y$ must be a sign, which concludes the proof of the claim.
\end{proof}

\section{Isometries between combinatorial spaces}\label{secIsom}

Let $T: X_{\F}^* \rightarrow X_{\G}^*$ be an operator between duals of combinatorial spaces $X_{\F}$ and $X_{\G}$, sending extreme points of the unit ball of $X_{\F}^*$ to extreme points of the unit ball of $X_{\G}^*$.
Our goal is to show that for any $i\in\N$, $Te^*_i$ is of the form $\sum_{j \in A_i} \theta_j^i e_j^*$, for finite subsets $A_i$ of $\N$ and sequences of signs $(\theta_j^i)_{j \in A_i}$ (Proposition \ref{permprop}).
Then, assuming $T$ is a surjective isometry, we prove that
$T$ is induced by a signed permutation.
% $Te^*_i=\pm e^*_{\pi(i)}$, where $\pi:\N\to\N$ is a permutation of the natural numbers (Theorem \ref{perm}). For the remainder of this section $T$ will be fixed, and for any  natural numbers $i$ and $j$ denote $\alpha^i_j:=Te^*_i(e_j)$ where $(e_j)_j$ denotes the basis of $X_{\G}$ and $(e_i^*)_i$ the basis of $X_{\F}^*$. In other words $\alpha^i_j$ is the coefficient of $e^*_j$ that appears in the expansion of $Te^*_i$ in the basis of $(e^*_n)_n$. 
%Denote by $Ext(X_{\F}^{*})$ the set of extreme points of the unit ball of $X_{\F}^{*}$. 
%We are going to use the following description of the extreme points of the dual ball (see %\cite{abh, gow}):

Since
$$
Ext(X_{\F}^{*})=\left\{\sum_{i\in F}\theta_i e^*_i : F\in\FM \text{ and }(\theta_i)_{i \in F} \text{ is a sequence of signs}\right\},
$$
it follows that if $x^*\in Ext(X_{\F}^{*})$, then for any $i\in\N$, $|x^*(e_i)| \in\{0, 1\}$. 

\begin{lemma} \label{onehalf}
Suppose $n\in\N$, $\{n\}\notin\FM$, and let $k\in Supp (Te^*_n)$ such that $|\alpha^n_k|\neq 1$. Then  $|\alpha^n_k|= \frac{1}{2}$, and for any $F\in\FM$ such that $n\in F$, there exists a unique $m\in F\setminus\{n\}$ such that $|\alpha^m_k|= \frac{1}{2}$. Moreover, $\alpha^j_k=0$ for all $j\in F\setminus\{n,m\}$.
\end{lemma}

\begin{proof}
Let $F\in\FM$ such that $n\in F$. If for all $j\in F\setminus\{n\}$, $k\notin Supp(Te^*_j)$, then we have
$$
T(\sum_{j\in F}e^*_j)(e_k)=\sum_{j\in F}Te^*_j(e_k)=\alpha^n_k
$$

Since $|\alpha^n_k| \notin\{0,1\}$, it follows that $T(\sum_{j\in F}e^*_j) \notin Ext(X_{\G}^{*})$, contradicting the fact that $\sum_{j\in F}e^*_j\in Ext(X_{\F}^{*})$ and $T$ preserves extreme points. Therefore there exists $m\in F$, $m \neq n$, such that  $k\in Supp(Te^*_m)$. Consider $(\theta_j)_{j\in F}$ a sequence of signs such that for any $j\in F$ we have $\theta_j\alpha^j_k\geq 0$. Since $\sum_{j\in F}\theta_j e^*_j\in Ext(X_{\F}^{*})$, it follows that $T(\sum_{j\in F}\theta_j e^*_j)\in Ext(X_{\G}^{*})$, hence

\begin{equation} \label{signs}
    T(\sum_{j\in F}\theta_j e^*_j)(e_k)=\sum_{j\in F}\theta_j Te^*_j(e_k)=\sum_{j\in F}\theta_j\alpha^j_k=1
\end{equation}

as all $\theta_j\alpha^j_k$ are non-negative and at least two, namely $\theta_n \alpha^n_k$ and  $\theta_m \alpha^m_k$, are positive.

On the other hand $\sum_{j\in F\setminus\{n\}}\theta_j e^*_j-\theta_n e^*_n$ is also an extreme point whose image by $T$ has real value in $e_k$, hence

\begin{equation} \label{minus}
\sum_{j\in F\setminus\{n\}}\theta_j Te^*_j(e_k)-\theta_n Te^*_n(e_k)=\sum_{j\in F\setminus\{n\}}\theta_j\alpha^j_k-\theta_n\alpha^n_k\in\{-1, 0, 1\}.
\end{equation}

From (\ref{signs}), (\ref{minus}), and the fact that $\theta_n \alpha^n_k$ and  $\theta_m \alpha^m_k$ are positive, it follows easily that
$\sum_{j\in F\setminus\{n\}}\theta_j\alpha^j_k-\theta_n\alpha^n_k=0$, and solving for $\theta_n\alpha^n_k$ we obtain that $\theta_n \alpha^n_k=\frac{1}{2}$. In a similar manner, reversing the roles of $n$ and $m$ we obtain that $\theta_m \alpha^m_k=\frac{1}{2}$ as well. Plugging these values into (\ref{signs}), and taking into account that all $\theta_j\alpha^j_k$ are non-negative, it follows that $\alpha^j_k=0$ for all $j\in F\setminus\{n,m\}$.
\end{proof}

\begin{lemma} \label{disjoint}
Suppose $n\in\N$, $\{n\}\notin\FM$, and let $k\in Supp (Te^*_n)$ such that $|\alpha^n_k|= 1$.
Then for any $F\in\FM$ such that $n\in F$, and for any $j\in F\setminus\{n\}$, $\alpha^j_k=0$.
\end{lemma}

\begin{proof}
Pick $F\in\FM$ such that $n \in F$, and consider the extreme points $\sum_{j\in F}\theta_j e^*_j$, where $(\theta_j)_{j\in F}$ are choices of signs such that $\theta_j \alpha_k^j $ is non-negative for all $j \in F$. Since $T(\sum_{j\in F}\theta_j e^*_j)$ is also an extreme point, it follows that

$$
\theta_n\alpha^n_k+\sum_{j\in F\setminus\{n\}}\theta_j\alpha^j_k
% \pm \theta_n+\sum_{j\in F\setminus\{n\}}\theta_j\alpha^j_k%
\in\{-1, 0, 1\},
$$
for all signs $(\theta_j)_{j\in F}$. Clearly this is only possible if $\alpha^j_k=0$ for all $j\in F\setminus\{n\}$.
\end{proof}

%\begin{lemma} \label{largenorm}
%Let $x^*=\sum_{j\in A}\theta_j e^*_j\in X^*_{\F}$, where $A\subseteq\N$, and $(\theta_j)$ are signs. If there exists $F\in\FM$ such that $F\subsetneq A$, then $\norm{x^*}>1$.
%\end{lemma}
%\begin{proof}
%Let $k\in A\setminus F$ and consider the vector $x\in X_{\F}$ defined as
%$$
%  x=\frac{1}{|F|}\sum_{j\in F}\theta_j e_j + \varepsilon \theta_k e_k, \mathrm{ with \hskip .3cm  } 0<\varepsilon<\frac{1}{|F|}.
%$$
%Clearly $\norm{x}=1$ and $\norm{x^*}\geq |x^*(x)|=1+\varepsilon$.
%\end{proof}

\begin{lemma}\label{largenorm}
For any finitely supported $x^*=\sum_{i \in A} \theta_i e_i^* \in X_\F^*$ with $|\theta_i|=1$, we have that $\Vert x^* \Vert = 1$ if and only if $A \in \F$ (and $\Vert x^*\Vert >1$ otherwise).
\end{lemma}
\begin{proof}
Since each $e_i \in X_\F$ has norm 1 and $|x^*(e_i)|=|\theta_i|=1$ for $i \in A$, clearly $\Vert x^*\Vert \geq 1$.

If $A\in \F$, then given $x = \sum_{j} \alpha_j e_j \in X_\F$ such that $\Vert x \Vert_\F =1$, we have that $\sum_{j \in A} |\alpha_j| \leq 1$, so that
$$|x^*(x)| \leq \sum_{i\in A} |\theta_i|\cdot |\alpha_i| = \sum_{i\in A} |\alpha_i| \leq 1.$$
Therefore, $\Vert x^* \Vert \leq 1$.

Conversely, if $A \notin \F$, then $|A| \geq 2$, so let $x=\frac{1}{|A|-1}\sum_{i\in A} \bar{\theta}_i e_i$ and notice that $\Vert x \Vert \leq 1$ and 
$$x^*(x) = \frac{1}{|A|-1}\sum_{i\in A} \theta_i \bar{\theta}_i = \frac{1}{|A|-1} \sum_{i \in A} |\theta_i|^2 = \frac{|A|}{|A|-1} > 1,$$
so that $\Vert x^* \Vert >1$.
\end{proof}

\begin{proposition} \label{permprop}
Let $T:X^*_{\F}\to X^*_{\G}$ preserve extreme points, where $\F, \G$ are regular families. Then the vectors $Te_i^*$, $i \in \N$, are of the form $\sum_{j \in A_i} \theta_j^i e_j^*$ for sets $A_i \in \G$ of $\N$ and signs $\theta_j^i$. 
\end{proposition}

\begin{proof}
With the previous notations, we are going to show first that for any $n\in \N$ and any $k\in\N$ we have $|\alpha^n_k| \neq\frac{1}{2}$. Fix $n\in\N$ arbitrary and note first that if $\{n\}\in\FM$ then $e^*_n$ is an extreme point, hence so is $Te^*_n$, and it follows that for any $k\in\N$ we have that $|Te^*_n(e_k)|=|\alpha^n_k|\in\{ 0, 1\}$. When $\{n\}\notin\FM$, assume towards a contradiction that there exists $k\in\N$ such that $|\alpha^n_k|=\frac{1}{2}$.  We are going to consider separately two cases: when $n$ belongs to a maximal set of size at least three, and when $n$ belongs only to maximal sets of size two.
\vskip .5cm
Case 1: There exists $F\in\FM$ such that $n\in F$ and $|F|\geq 3$.

In this case, construct a sequence as in Lemma \ref{spreading}, starting with $G_1:=F\setminus\{n\}$. It follows that for each $i\in\N$, $|G_i|\geq 2$. Since $|\alpha^n_k|=\frac{1}{2}$, and for any $i\in\N$ we have that $G_i\cup\{n\}\in\FM$ and $|G_i|\geq 2$, from Lemma \ref{onehalf} we conclude that there exists a sequence $p_i\in G_i$, $i\in\N$, such that $\alpha^{p_i}_k=0$.  From Lemma \ref{propRegular}, there is $E\in\FM$ such that $E \subseteq \{n, p_1, p_2, \dots\}$ and $n = \min E$. However,
$$
|T(\sum_{j\in E}e^*_j)(e_k)|=|\alpha^n_k+\sum_{p_i\in E\setminus\{n\}}\alpha^{p_i}_k|=|\alpha^n_k|=\frac{1}{2},
$$
contradicting the fact that $T(\sum_{j\in E}e^*_j)$ is an extreme point.

\vskip .5cm

Case 2: For any $F\in\FM$ such that $n\in F$, $|F|=2$.

Assume there exists $m>n$ such that $\{n,m\}\in\FM$ and $m$ belongs to a maximal set of size at least $3$. Then it follows from Lemma \ref{onehalf}  that $|\alpha^m_k|=\frac{1}{2}$, and from Case 1, applied to $m$, we obtain a contradiction. Hence, we may also assume that for any $m>n$, such that $\{n,m\}\in\FM$, $m$ only belongs to maximal sets of size $2$. Construct a sequence of sets $n<G_1<G_2<\dots$ as in Lemma \ref{spreading}. Then we must have that each $G_i$ is a singleton, so we obtain a sequence $n<q_1<q_2<\dots$ such that $\{n,q_i\}\in\FM$ for all $i\in\N$. Also, from Lemma \ref{onehalf}, we conclude that $|\alpha^{q_i}_k|=\frac{1}{2}$ for all $i\in\N$. From spreading we have that $\{q_i, q_j\}\in\F$ for all $i<j$, and since no $q_i$ belongs to a maximal set of size at least 3, it follows that actually $\{q_i, q_j\}\in\FM$ for all $i<j$. 

%Pick a subsequence $(p_i)_i$ of $(q_i)_i$ such that $\alpha^{p_i}_k=\frac{1}{2}$ for all %$i\in\N$. The argument works similarly if all $\alpha^{p_i}_k=-\frac{1}{2}$, and obviously we %can pick at least one of the two choices. 
For each $i$ write
$T(e_{q_i}^*)=\frac{1}{2}\epsilon_{i} e_k^*+\frac{1}{2}y_i^*+z_i^*$, where $\epsilon_i$ is a sign, the three vectors are disjointly supported (possibly $y_i^*$ or $z_i^*$ is $0$), and $y_i^*$ and $z_i^*$ only have coordinates of modulus $1$ on their support.
In the complex case, we note that $T(\theta_i e_{q_i}^*+\theta_j e_{q_j}^*)$ being an extreme points for all signs $\theta_i, \theta_j$ contradicts the fact that its $k$-coordinate is $\frac{1}{2}(\theta_i \epsilon_i +\theta_j \epsilon_j)$, which can assume values of modulus different from $0$ and $1$. In the real case, passing to a subsequence we may assume that $\epsilon_i$ is constant, and without loss of generality, equal to $1$; from the fact that for $i \neq j$, $T(e_{q_i}^*\pm e_{q_j}^*)$ must be an extreme point and therefore does not have $\pm \frac{1}{2}$ coordinates, we deduce that the support of $y_i^*$ is some finite set $C$ independent of $i$ and that $z_i^*$ is disjointly supported from $k$, from  $y_i^*$ and from all other $z_j^*$. Since $C$ is finite we find $i \neq j$ such that $y_i^*=y_j^*$, and we compute
$$T(e_{p_i}^*+e_{p_j}^*)=e_k^*+y_i^*+z_i^*+z_j^*,$$
and
$$T(e_{p_i}^*-e_{p_j}^*)= +z_i^*-z_j^*,$$
and so the second vector has support strictly included in the support of the first one. But this contradicts that both must belong to $\GM$.

The fact that $A_i \in \G$ follows from Lemma \ref{largenorm}.
\end{proof}

\begin{theorem} \label{perm}
Let $T:X^*_{\F}\to X^*_{\G}$ be an isometry, where $\F, \G$ are regular families. Then there exists a permutation $\pi:\N\to\N$ and a sequence of signs $(\theta_i)_i$ such that $Te^*_i=\theta_i e^*_{\pi(i)}$ for all $i\in\N$.
\end{theorem}

\begin{proof}
Note first that if  $\{i,k\} \in \F$ and $i \neq k$, then $Supp Te_i^* \cap Supp Te_k^* = \emptyset$. Indeed, from Proposition \ref{permprop}, $|(Te_i^*)(e_k)|, |(Te_k^*)(e_k)| \in \{0,1\}$. Hence, for a given $j$, let $\theta_i$ and $\theta_k$ be signs such that $\theta_i(Te_i^*)(e_j) = |(Te_i^*)(e_j)|$ and $\theta_k (Te_k^*)(e_j) = |(Te_k^*)(e_j)|$. Since $\Vert \theta_i e_i^* + \theta_k e_k^* \Vert = 1$, we have that $\theta_i(Te_i^*)(e_j) + \theta_k(Te_k^*)(e_j) \leq \Vert\theta_i  Te_i^* + \theta_k Te_k^* \Vert = 1$, so that $|(Te_i^*)(e_k)|, |(Te_k^*)(e_k)|$ cannot be both $1$. This guarantees that the supports are disjoint. Of course, a similar fact holds true for a pair $\{j,l\}\in\G$ and $T^{-1}$. 

We are going to show that for any $n\in\N$, the support of $Te_n^*$ is a singleton. From the fact that $T$ is a bijection, the conclusion of the theorem follows immediately. 

Fix $n\in\N$ arbitrary, and from Proposition \ref{permprop} it follows we can write 
$$
Te_n^*=\sum_{i\in A}\theta_i e^*_i
$$
where $\theta_i$ are signs and $A\in\G$. From the hereditary property we have that for any $j,l\in A$, $j\neq l$, the set $\{j,l\}\in G$. Therefore, from the remark at the beginning of the proof, we conclude that $T^{-1}e^*_j$ and  $T^{-1}e^*_l$ have disjoint support. Hence
$$
\bigcup_{i\in A} Supp( T^{-1}e^*_i)=Supp \left(\sum_{i\in A}\theta_i T^{-1}e^*_i\right)=Supp\left( T^{-1}\sum_{i\in A}\theta_ie^*_i\right)=\{n\}
$$
Therefore  $|A|\leq |\bigcup_{i\in A} Supp( T^{-1}e^*_i)|=1$, and from this it follows immediately that $A$ is a singleton, as claimed. This finishes the proof. 
\end{proof}

The following example shows that a bounded operator that sends extreme points to extreme points and vectors of disjoint support to vectors of disjoint support is not necessarily given by a permutation of the basis. 

\begin{example} The map $T$ defined on the dual of the Schreier space by 
$T(e_n^*)=e_{2n}^*+e_{2n+1}^*$, sends extreme points to extreme points, sends disjoint supports to disjoint supports, but is not induced by a signed permutation.
\end{example}

\begin{proof} For $n \geq 1$ any sum of $e_i^*$ supported on some $F$ such that $|F|=\min F=n$ has image supported on some $F'$ such that $|F'|=2n=\min |F'|$.\end{proof}

\begin{corollary} 
Assume $\F, \G$ are regular families. Then TFAE:
\begin{enumerate}[(i)]
    \item $X_{\F}$ and $X_{\G}$ are isometric;
    \item $X^*_{\F}$ and $X^*_{\G}$ are isometric;
    \item There is a permutation $\pi$ of $\N$ such that $\GM=\{\pi(F): F\in \FM\}$;
    \item There is a permutation $\pi$ of $\N$ such that $\G=\{\pi(F): F\in \F\}$.
\end{enumerate}
\end{corollary}
\begin{proof}
(i) implies (ii) for any two Banach spaces. 

It follows from Theorem \ref{perm} that if $T:X^*_{\F}\to X^*_{\G}$ is an isometry, then it is induced by a signed permutation
% $\pi:\N\to\N$ such that $Te^*_i=\pm e^*_{\pi(i)}$ for all $i\in\N$. 
Since $T$ takes extreme points to extreme points, in particular we get that $F \in \FM$ if and only if $\pi(F) \in \GM$. 

(iii) trivially implies (iv).

Finally, if $\pi$ is a permutation of $\N$ such that $\G=\{\pi(F): F\in \F\}$, it is easy to see that $T: X_\F \rightarrow X_\G$ defined by $T(\sum_i \lambda_i e_i) = \sum_i \lambda_i e_{\pi(i)}$ is an onto isometry.
\end{proof}

%Note that $X_\F$ and $X_\G$ can be isometric for families $\F$ and $\G$ that are distinct, as the following example shows.
%
%\begin{example}
%Consider the following families $\FM$ and $\GM$
%$$
%\FM=\{\{1\}, \{2,3\}, \{2,a,b\}, \{3,a, b, c\}, %\{a,b,c,d\}: 4\leq a<b<c<d\}
%$$
%$$
%\GM=\{\{1\}, \{2,3\}, \{3,a,b\}, \{2,a, b, c\}, %\{a,b,c,d\}: 4\leq a<b<c<d\}
%$$

%Then $T:X_{\F}\to X_{\G}$ defined as %$Te_i=e_{\pi(i)}$ where $\pi(2)=3$, $\pi(3)=2$, and %$\pi(j)=j$ for any $j\notin\{2,3\}$ is an isometry %between  $X_{\F}$ and $X_{\G}$, yet $\F\neq\G$. 
%\end{example}

\section{Isometries of Schreier spaces}\label{secSchreier}

\begin{definition}
Given a countable ordinal $\alpha$, we define the Schreier family of order $\alpha$ inductively as follows:
\begin{itemize}
    \item $\mathcal{S}_1 = \mathcal{S}$;
    \item $\mathcal{S}_{\alpha+1} = \{\cup_{j=1}^k E_j: E_j \in \mathcal{S}_\alpha \text{ and } \{\min E_j: 1 \leq j \leq k\} \in \mathcal{S}\}\cup \{ \emptyset\}$;
    \item $\mathcal{S}_\alpha = \{F \in [\N]^{<\omega}: F \in \mathcal{S}_{\alpha_n} \text{ for some }n \leq \min F \} \cup \{\emptyset\}$, if $\alpha$ is a limit ordinal and $(\alpha_n)_n$ is a fixed increasing sequence of ordinals converging to $\alpha$. 
\end{itemize}
\end{definition}

Note that the sequence of Schreier families $(\Sh_\alpha)_{\alpha<\omega_1}$ depends on the choice of the sequences $(\alpha_n)_n$ converging to each limit ordinal $\alpha$. It is a well-known fact \cite{at} that Schreier families are regular families, so that we may apply the results from the previous section to these families.

\begin{lemma} \label{Shreiermax}
Let $E$ and $F$ be two maximal sets in $\Sh_{\alpha}$, where $\alpha<\omega_1$. If $F$ is a spreading of $E$ then $\min E=\min F$.
\end{lemma}

\begin{proof}
We are going to prove the statement by transfinite induction. It clearly holds true for $\Sh_1$, and assuming it holds for $\Sh_{\beta}$, for all $\beta<\alpha$, we will prove it for $\Sh_{\alpha}$.

Case 1: $\alpha$ is a successor ordinal, hence $\alpha=\beta +1$ for some $\beta<\omega_1$.

Let $E=\cup_{j=1}^{k}E_j$ for some $E_j \in \mathcal{S}_\beta$, $E_j < E_{j+1}$ and $\{\min E_j: 1 \leq j \leq k\} \in \mathcal{S}$. Since $E$ is maximal, $k= \min E_1 = \min E$. Let $\sigma: E \rightarrow F$ be the order-preserving bijection and, since $F$ is a spreading of $E$, then $\sigma(n) \geq n$ for every $n \in E$. In particular, $F_j:= \sigma(E_j) \in \mathcal{S}_\beta$, as $\mathcal{S}_\beta$ is spreading, and $\{\min F_j: 1 \leq j \leq k \} \in \mathcal{S}$, as $\mathcal{S}$ is spreading. Since $F$ is maximal, we get that $\min F = \min F_1 = k = \min E$.

Case 2: $\alpha$ is a limit ordinal.

Let $n<\min E$ be such that $E \in \mathcal{S}_{\alpha_n}$. Since $F$ is a spreading of $E$, we get that $F \in \mathcal{S}_{\alpha_n}$. The maximality of $E$ and $F$ in $\mathcal{S}_\alpha$ implies, by Lemma \ref{propRegular} of $\mathcal{S}_{\alpha_n}$, that they are both also maximal in $\mathcal{S}_{\alpha_n}$. By the inductive hypothesis, we get that $\min E = \min F$. 
\end{proof}

\begin{theorem}
Let $T:X^*_{\Sh_\alpha}\to X^*_{\Sh_\alpha}$ be an isometry. Then there exist a sequence of signs $(\theta_i)_i$ such that for any $i\in\N$, $Te^*_i=\theta_i e^*_i$.
\end{theorem}
\begin{proof}
Fix $T:X^*_{\Sh_\alpha}\to X^*_{\Sh_\alpha}$, and from Theorem \ref{perm} we have that there exists a permutation $\pi:\N\to\N$ and sequence of signs $(\theta_i)_i$ such that  $Te^*_i=\theta_i e^*_{\pi(i)}$. Assume towards a contradiction that the conclusion is not true, and let $k_0$ be the smallest integer such that $p_0:=\pi(k_0)\neq k_0$. Note that from the proof of Theorem \ref{perm} follows that $\pi$ sends maximal singletons to maximal singletons, and since $\{1\}$ is the only maximal singleton in $\Sh_{\alpha}$ we have that $k_0>1$ and $\{k_0\}\notin\SM_{\alpha}$. From the minimality of $k_0$ we also have that $p_0>k_0$.

Pick $k_1>k_0$ such that $p_1:=\pi(k_1)\geq k_1$ and $p_1>p_0$. Note that we can always do that, since any permutation will contain an increasing sequence, and we can go far enough along that sequence to pick a suitable $k_1$. Continuing in this manner we construct infinite sequences $\{k_0, k_1, k_2 \dots\}$ and
$\{p_0, p_1, p_2 \dots\}$ such that $p_i:=\pi(k_i)$ and $\{p_0, p_1, p_2 \dots\}$ is a spreading of $\{k_0, k_1, k_2 \dots\}$. From the barrier property, we can find an initial segment $E\subset\{k_0, k_1, k_2 \dots\}$ such that $E\in\SM_{\alpha}$. Since $T$ sends extreme points to extreme points, and $E\in\SM_{\alpha}$, it follows that $\pi(E)\in\SM_{\alpha}$ as well. Hence, from Lemma \ref{Shreiermax} we must have that $\min E=\min \pi(E)$. That is, $k_0=p_0$, which contradicts the initial assumption. This finishes the proof.
\end{proof}

\section{Final remarks}\label{secFinal}

Theorem \ref{perm} guarantees that all the isometries of a combinatorial space or its dual are determined by a permutation of the elements of the basis and a change of signs. Natural and general questions which remain open are the following:

\begin{question}
Given a regular family $\F$, what are the permutations of the basis which induce an isometry of $X_\F$? 
\end{question}

\begin{question}
For which combinatorial spaces can we explicitly describe the group of its isometries?
\end{question}

In Section \ref{secSchreier} we described the group of the isometries on Schreier spaces, showing that the identity is the only permutation  allowed. The following example illustrates an intermediate situation where more permutations are allowed, though not all of them:

\begin{example}\label{sequence}
Given an increasing sequence $(k_n)_n$ such that $k_0=0$, let
$$\F = \{F \in [\N]^{<\omega}: |F| \leq n, \text{ where }k_{n-1} \leq \min F <k_n\}.$$
Then $T: X_\F^* \rightarrow X_\F^*$ is an isometry if, and only if, there is a permutation $\pi$ of $\mathbb{N}$ such that for all $n\in \mathbb{N}$, $\pi(I_n)=I_n$ and $T(e_n^*)= \pm e_{\pi(n)}^*$, where $I_n=[k_n, k_{n+1}[$.
\end{example}
\begin{proof}
It is easy to see that $\F$ is hereditary and spreading, and to prove compactness one should follow similar arguments as to the Schreier families. Moreover, one easily sees that $\FM$ is a barrier.

Given an isometry $T: X_\mathcal{F}^* \rightarrow X_{\mathcal{F}}^*$, by Theorem \ref{perm}, there exists a permutation $\pi: \N \rightarrow \N$ such that $T(e_n^*)= \pm e_{\pi(n)}^*$ for every $n \in \mathbb{N}$. Note that from the proof of Theorem \ref{perm} it follows that $\pi(F) \in \FM$ iff $F \in \FM$. On the other hand, $F \in \FM$ iff $|F|=n$ for the unique $n$ such that $k_{n-1} \leq \min F < k_n$. It follows easily that $\pi(I_n)=I_n$. 

Conversely, given a permutation $\pi$ of $\N$ such that $\forall n\in \mathbb{N} \ \pi(I_n)=I_n$, we have that $\mathcal{F} = \{\pi(F): F \in\mathcal{F}\}$. Hence, if $T(e_n^*)= \pm e_{\pi(n)}^*$ for every $n \in \mathbb{N}$, one can take $T$ to be the linear operator that takes $e_n^*$ to $e_{\pi(n)}^*$ and is is easy to see that $T$ is an isometry.
\end{proof}

%\section{Further examples}

%\begin{example} Assume $\F=\{\{n\}, n\in \N\}$ and $\G$ any regular family. Let $T$ be a map defined by $T(e_n^*)=t_n$, where $(t_n)_n$ is a sequence of extreme points of $X_{\G}^*$ (possibly repeating values). Then $T$ extends to a linear operator from $X_{\F}^*$ to $X_{\G}^*$ sending extreme points to extreme points. \end{example}

%\begin{proof} The basis $e_n^*$ is equivalent to the canonical basis of $\ell_1$. \end{proof}

\end{document}